\documentclass[preprint,12pt]{elsarticle}
\usepackage{graphicx}
\usepackage[left=3cm,right=1.5cm, top=2cm,bottom=2cm,bindingoffset=0cm]{geometry}
\usepackage{amssymb}
\usepackage{amsthm}
\usepackage{cmap}

\usepackage{amsmath}% http://ctan.org/pkg/amsmath

\biboptions{sort&compress}
\newcommand{\lb}{\linebreak}
\newcommand{\sysn}{\left\{\begin{array}{rcl}}
\newcommand{\sysk}{\end{array}\right.}

\newcommand{\ingrw}[2]{\includegraphics[width=#1mm]{#2}}

\newcommand{\thetacl}[1]{\overline{#1}^\theta}
\newcommand{\deltacl}[1]{\overline{#1}^\delta}
\newcommand{\sthetacl}[1]{\overline{#1}^\Theta}

\newtheorem{theorem}{Theorem}[section]
\newtheorem{lemma}[theorem]{Lemma}

\theoremstyle{example}
\newtheorem{example}[theorem]{Example}

\newtheorem{proposition}[theorem]{Proposition}
\theoremstyle{definition}
\newtheorem{definition}[theorem]{Definition}
\newtheorem{corollary}[theorem]{Corollary}

\journal{...}

\begin{document}

\begin{frontmatter}

%% Title, authors and addresses

%% use the tnoteref command within \title for footnotes;
%% use the tnotetext command for the associated footnote;
%% use the fnref command within \author or \address for footnotes;
%% use the fntext command for the associated footnote;
%% use the corref command within \author for corresponding author footnotes;
%% use the cortext command for the associated footnote;
%% use the ead command for the email address,
%% and the form \ead[url] for the home page:
%%
\title{Various types of completeness in topologized semilattices}

\author[affil1]{Konstantin Kazachenko}

\address[affil1]{Krasovskii Institute of Mathematics and Mechanics, Yekaterinburg, Russia}

\ead[affil1]{voice1081@gmail.com}

\author[affil2]{Alexander V. Osipov}

\address[affil2]{Krasovskii Institute of Mathematics and Mechanics, \\ Ural Federal
 University, Ural State University of Economics, Yekaterinburg, Russia}

\ead[affil2]{OAB@list.ru}
%% use optional labels to link authors explicitly to addresses:
%% \author[label1,label2]{<author name>}
%% \address[label1]{<address>}
%% \address[label2]{<address>}

\begin{abstract}
A topologized semilattice $X$ is called {\it complete} if each
non-empty chain $C\subset X$ has $\inf{C}$ and $\sup{C}$ that
belong to the closure $\overline{C}$ of the chain $C$ in $X$.
 In this paper, we introduce various
concepts of completeness of topologized semilattices in the
context of operators that generalize the closure operator, and
study their basic properties. In addition, examples of specific
topologized semilattices are given, showing that these classes do
not coincide with each other. Also in this paper, we prove
theorems that allow us to generalize the available results on
complete semilattices endowed with topology.

%The obtained results allow us to expand the study the absolute
%closure property of complete topologized semilattices.

\end{abstract}

\begin{keyword} topologized semilattice \sep complete semilattice \sep $\theta$-closed set \sep $\delta$-closed set \sep $H$-set
%% keywords here, in the form: keyword \sep keyword

%% MSC codes here, in the form: \MSC code \sep code

\MSC[2010] 06B30 \sep 06B35 \sep 54D55
%% or \MSC[2008] code \sep code (2000 is the default)

\end{keyword}

\end{frontmatter}

%%
%% Start line numbering here if you want
%%
% \linenumbers

%% main text

\section{Introduction}
The use of various algebraic structures, additionally endowed with
the structure of a topological space, has long established itself
as a convenient and powerful tool in various fields of modern
mathematics. This phenomenon motivates the fundamental study of
the properties of these objects.

When studying spaces that have some additional structure
consistent with the topology, quite often the concept of
completeness naturally arises, as some internal property of these
objects. In most cases, completeness can also be described as an
external property, and often it is associated with the concept of
absolute closedness, understood in a suitable sense. For example, a
metric space $X$ is complete if and only if it is closed in every
metric space $Y$ containing $X$ as a metric subspace. A uniform
space $X$ is complete if and only if it is closed in every uniform
space $Y$ containing X as a uniform subspace. A topological group
$X$ is complete if and only if $X$, together with its two-sided
uniform structure, is a complete uniform space, and so on. The
completeness of semilattices is a well-studied algebraic property,
which generalizes quite naturally (using the closure operator in a
topological space) to semilattices endowed with a
topological structure.

It should be noted that one of the first mathematicians who studied the absolute closedness of various topologized algebraic objects (including semilattices) was O. V. Gutik (see for example \cite{gutik1, gutik2, gutik3})

The question of the closedness of the images of complete
topologized semilattices under continuous homomorphisms in
Hausdorff semitopological semilattices was first raised by T.
Banakh and S. Bardyla in \cite{BaBa1} and is currently solved
positively for some special cases. Historically, the first results
in this direction belong to J. W. Stepp; in particular, he proved
that semilattices having only finite chains are always closed in
every  topological semilattice containing it as a subsemilattice
\cite{stepp}. The above circumstances motivate the study of the
notion of completeness of topologized semilattices in the context
of operators that generalize the closure operator in a topological
space.

The objectives of this paper are to determine the corresponding
classes of semilattices endowed with topology, to study their
basic properties, to construct examples showing that these classes
do not coincide, and to generalize the well-known theorems on the
closure of images of complete semilattices under continuous
homomorphisms.

\section{The completeness of topologized semilattices}

A {\it semilattice} is any commutative semigroup of idempotents
(an element $x$ of a semigroup is called an {\it idempotent} if
$xx = x$).

A semilattice endowed with a topology is called a {\it topologized
semilattice}. A topologized semilattice $X$ is called a {\it
(semi)topological semilattice} if the semigroup operation $X\times
X \rightarrow X$, $(x,y)\mapsto xy$, is (separately) continuous.

It is well known that semilattices can be viewed as partially ordered
sets, namely: in every semilattice $X$ we can consider the
following order relation $\leq$: $x \leq y \leftrightarrow
xy=x=yx$. Endowed with this partial order, the semilattice is a
{\it poset}, i.e., partially ordered set. It is easy to see that
the element $xy$ is, in the sense of a given order, an infimum (a
greatest lower bound) of the elements $x$ and $y$. Conversely, if in a partially
ordered set $(X, \leq)$ each pair of elements has a greatest lower bound, then $X$ together with the operation of
taking the infimum is a semilattice.

A subset $D$ of a poset $(X,\leq)$ is called

$\bullet$  a {\it chain} if any elements $x,y\in D$ are comparable
in the sense that $x\leq y$ or $y\leq x$. This can be written as
$y\in {\updownarrow}x$ where ${\uparrow}x:=\{y\in D: x\leq y\}$,
${\downarrow}x:=\{y\in D: y\leq x\}$, and $\updownarrow
x:=({\uparrow}x)\cup ({\downarrow}x)$;

$\bullet$ {\it up-directed} if for any $x,y\in D$ there exists
$z\in D$ such that $x\leq z$ and $y\leq z$;

$\bullet$ {\it down-directed} if for any $x,y\in D$ there exists
$z\in D$ such that $z\leq x$ and $z\leq y$.

It is clear that each chain in a poset is both up-directed and
down-directed.

 A semilattice $X$ is called {\it chain-finite} if each chain in
$X$ is finite. A semilattice is called {\it linear} if it is a
chain in itself.

In \cite{stepp} Stepp proved that for any homomorphism
$h:X\rightarrow Y$ from a chain-finite semilattice to a Hausdorff
semitopological semilattice $Y$, the image $h(X)$ is closed in
$Y$.

In \cite{BaBa1} Banakh and Bardyla improved result of Stepp by
proving that for any homomorphism $h:X\rightarrow Y$ from a
chain-finite semilattice to a Hausdorff semitopological
semilattice $Y$, the image $h(X)$ is closed in $Y$.

The notion of completeness of semilattices is a well-known
algebraic property and is naturally transferred to topologized
semilattices: a topologized semilattice $X$ is called {\it
complete} if each non-empty chain $C\subset X$ has $\inf{C}$ and
$\sup{C}$ that belong to the closure $\overline{C}$ of the chain
$C$ in $X$.

A Hausdorff space $X$ is said to be {\it $H$-closed} if it is
closed in every Hausdorff space in which it can be embedded.

 Complete topologized semilattices play an important role in the
theory of (absolutely) $H$-closed semilattices, see [1-8]. By
\cite{BB1}, a {\it Hausdorff semitopological semilattice $X$ is
complete if and only if each closed chain in $X$ is compact if and
only if for any continuous homomorphism $h: S\rightarrow Y$ from a
closed subsemilattice $S\subset X$ to a Hausdorff topological
semigroup $Y$ the image $h(S)$ is closed in $Y$}.

A topologized semilattice $X$ is called

$\bullet$  {\it ${\uparrow}{\downarrow}$-closed} if for each $x
\in X$ the sets ${\uparrow}x$ and ${\downarrow}x$ are closed;

$\bullet$ {\it chain-compact} if each closed chain in $X$ is
compact.

\medskip

On each topologized semilattice we shall consider weaker
topologies:

- the {\it weak $chain^\bullet$-topology}  generated by the
subbase consisting of complements to closed chains in $X$,

- the {\it weak$^\bullet$ topology $\mathcal{W}^\bullet_X$},
generated by the subbase consisting of complements to closed
subsemilattices in $X$.

\medskip

 A topologized semilattice $X$ is called

$\bullet$ {\it $chain$-compact} if each closed chain in $X$ is
compact,

 $\bullet$ {\it weak $chain^\bullet$ compact} if its weak chain$^\bullet$ topology is compact,

$\bullet$ {\it $\mathcal{W}^\bullet_X$-compact} if its
weak$^\bullet$ topology $\mathcal{W}^\bullet_X$ is compact.

\medskip

The weak$^\bullet$-topology $\mathcal{W}^\bullet_X$ was introduced
and studied in \cite{BaBa4}. According to Lemmas 5.4, 5.5 of
\cite{BaBa4}, for any topologized semilattice we have the
implications:

complete $\Rightarrow$  $\mathcal{W}^\bullet_X$-compact
$\Rightarrow$ chain-compact.

\medskip
By Theorem 5.4 in \cite{BaBa5}, a chain-compact
${\uparrow}{\downarrow}$-closed topologized semilattice is
complete.

\medskip

Note that the weak chain$^\bullet$ topology of a topologized
semilattice $X$ is obviously contained in the weak$^\bullet$-
topology, which immediately implies:

\begin{lemma} A complete topologized semilattice $X$ is weak
$chain^\bullet$ compact.
\end{lemma}

It is also easy to prove the following statement.

\begin{lemma} A weak
$chain^\bullet$-compact topologized semilattice $X$ is
chain-compact.
\end{lemma}

\begin{proof}
Let $C$ be a closed chain in $X$. Consider a centered family
$\lbrace F_\alpha \rbrace$ of closed subsets of $C$. Then every
set $F_\alpha$ is also a chain and is closed in $X$; this means
that $F_\alpha$ is also closed in the weak
chain$^\bullet$-topology on $X$. Since $X$ is a weak
chain$^\bullet$ compact semilattice, $\bigcap\limits_{\alpha \in
\Lambda} F_\alpha \neq \emptyset$.
\end{proof}

From this we obtain the following theorem, which is necessary for
the study of other types of completeness of semilattices,
discussed in the following chapters.

\begin{theorem}\label{th23}
For a ${\uparrow}{\downarrow}$-closed topologized semilattice $X$,
the following statements are equivalent:

1) $X$ is complete;

2) $X$ is weak $chain^\bullet$ compact;

3) $X$ is chain compact.
\end{theorem}

A {\it multi-valued map} $\Phi : X \multimap Y$ between sets $X$,
$Y$ is a function assigning to each point $x\in X$ a subset
$\Phi(x)$ of $Y$.  The {\it image} of any set $A \subset X$ under a
multi-valued map $\Phi$ is called the set $\Phi(A) =
\bigcup\limits_{x \in A} \Phi(x)$, the {\it preimage} of any set
$B \subset Y$ is the set $\Phi^{-1}(B) = \lbrace x \in X : \Phi(x)
\cap B \neq \emptyset \rbrace$.
  A multi-valued map $\Phi : X \multimap Y$ between semigroups is
called a {\it multimorphism} if $\Phi(x)\Phi(y) \subset \Phi(xy)$
for any elements $x, y \in X$. Here $\Phi(x)\Phi(y):=\{ab: a\in
\Phi(x), b\in \Phi(y)\}$.

A multi-valued map $\Phi : X \multimap Y$ between topological
spaces is called {\it upper semicontinuous} if for any closed
subset $F\subset Y$ the preimage $\Phi^{-1}(F)$ is closed in $X$.

A subset $F$ of a topological space $X$ is called $T_1$-{\it
closed} (resp. $T_2$-closed) in $X$ if each point $x\in X\setminus
F$ has a (closed) neighborhood, disjoint with $F$.

A multimorphism $\Phi : X \multimap Y$ is called a {\it
$T_i$-multimorphism} for $i\in \{1,2\}$ if for any $x\in X$ the
set $\Phi(x)$ is $T_i$-closed in $Y$.

It was shown in [2] that the completeness of semilattices is
preserved by images under continuous homomorphisms.

 Note that the map $\Phi : X \multimap Y$ where $\Phi(x) = Y$ (between
 semilattices $X$ and $Y$) always is an upper semicontinuous
$T_1$ multimorphism; it shows the images of semilattices under
maps of this type do not preserve completeness. However, a
positive result was achieved by imposing additional algebraic
constraints, which will be shown in Theorem \ref{th25}. To prove
it, we need the following simple proposition.

\begin{proposition}
A ${\uparrow}{\downarrow}$-closed topologized semilattice $X$ is
complete if and only if each non-empty closed chain $C$ contains
$\inf C$ and $\sup C$.
\end{proposition}

\begin{proof} Necessity obviously follows from the definition of completeness.

 Let $C \subset X$ be a non-empty chain. By Lemma 4.2, proved in \cite{BaBa5}, the set $\overline{C}$ is also a
chain. Let $a$ be a the smallest element of $\overline{C}$; the
inclusion $C \subset \overline{C}$ implies that $a$ is the lower
bound of the set $C$. If we assume that there is a lower bound $c$
of chain $C$ such that $c \not\leq a$, then $a \not\in {\uparrow}c
\supset \overline{C}$, a contradiction. Hence, $a = \inf{C}$.
Similarly, we can show that the largest element $b$ of the chain
$\overline{C}$ is  $\sup C$, which completes the proof.
\end{proof}

\begin{theorem}\label{th25} Let $X$ be a complete topologized semilattice, $Y$
be a $\uparrow\downarrow$-closed topologized semilattice and $\Phi
: X \multimap Y$ be an upper semicontinuous $T_1$ multimorphism
such that for any two points $x, y \in X$ inequality $x \leq y$
implies $\Phi(x) \cap {\uparrow}\Phi(y) \subset \Phi(y)$. Then the
semilattice $\Phi(X)$ is complete if and only if the semilattice
$\Phi(x)$ is complete for each $x \in X$.
\end{theorem}

\begin{proof}
Let us first prove sufficiency. Since $\Phi$ is a $T_1$
multi-valued map, $\Phi(x)$ is closed for every $x \in X$ and,
hence, it is complete.

Let $\Phi(x)$ be a complete semilattice for each $x \in X$ and let
$C \subset Y$ be a closed chain. Since the semilattice $Y$ is
$\uparrow\downarrow$-closed, it is sufficient to show that $C$
contains the largest and smallest elements. Note that the semilattice
$\Phi^{-1}(C)$ is closed in $X$ and, hence, it is complete.
Consider the map $\Phi_C : \Phi^{-1}(C) \multimap C$ such that
$\Phi_C(x) = \Phi(x) \cap C$. We show that $\Phi_C$ is an upper
semicontinuous $T_1$ multimorphism. Since $\Phi_C(x)\Phi_C(y):=
(\Phi(x) \cap C)(\Phi(y) \cap C) \subset \Phi(xy) \cap C =
\Phi_C(xy)$ (inclusion holds due to the fact that
$\Phi$ is a multimorphism and $C$ is a semilattice), then $\Phi_C$
is a multimorphism.

Let $F \subset C$ be a closed subset of $C$. Then $F$ is closed in
$Y$ and $\Phi_C^{-1}(F) = \lbrace x \in \Phi^{-1}(C) : \Phi_C(x)
\cap F \neq \emptyset \rbrace = \lbrace x \in \Phi^{-1}(C) :
\Phi(x) \cap C \cap F \neq \emptyset \rbrace = \Phi^{-1}(C)$, that
is $\Phi(C)$ is upper semi-continuous. Finally, $\Phi_C(x) =
\Phi(x) \cap C$ is closed in $\Phi(x)$ and therefore is complete as closed
subsemilattice of semilattice $\Phi(x)$.

For each $c \in C$ consider a closed semilattice $S_c =
\Phi_C^{-1}({\uparrow c})$. Since $X$ is complete, $\inf{S_c} \in
\overline{S_c} = S_c$. Take $y \in C$ such that $c \in \Phi_C(y)$.
Since $\Phi_C(\inf{S_c}) \cap {\uparrow}c \neq \emptyset$, then $c
\in \Phi_C(\inf{S_c})\Phi_C(y) \subset \Phi_C(y\inf{S_c}) =
\Phi(\inf{S_c})$. Note that for $c_1, c_2 \in C$ the inequality
$c_1 \leq c_2$ implies the inclusion $S_{c_2} \subset S_{c_1}$,
whence it follows that $\inf{S_{c_1}} \leq \inf{S_{c_2}}$,i.e. the
set $S = \lbrace S_c \rbrace_{c \in C}$ is a chain. By the completeness of $X$, $\inf{S} \in \overline{S} \subset \Phi^{-1}(C)$.
Since the chains $\Phi_C(\inf{S})$ and $\Phi_C(\sup{S})$ are
complete and closed, they contain the largest and the smallest
elements. Let $a = \min{\Phi_C(\inf{S})}$, $b =
\max{\Phi_C(\sup{S})}$. We show that $a$ and $b$ are the smallest
and largest elements of $C$, respectively. Indeed, suppose that the chain
 $C$ contains the element $c > b$.
Then $\inf{S_c} \leq \sup{S}$ and, by the condition, $c\in
\Phi_C(\inf{S_c}) \cap {\uparrow}\Phi_C(\sup{S}) \subset
\Phi_C(\sup{S})$. Since $b$ is a largest element of
$\Phi_C(\sup{S})$, $c \leq b$.

Now suppose that there is an element $c < a$.  Then $c \in
\Phi_C(\inf{S_c})\Phi_C(\inf{S}) \subset \Phi_C(\inf{S})$ and $c
\geq a$ because of  $a$ is the smallest element of
$\Phi_C(\inf{S})$. The resulting contradictions complete the
proof.
\end{proof}

Note important special case of Theorem \ref{th25}.

\begin{corollary}
Let $X$ be a complete topologized semilattice, let $Y$ be a
$\uparrow\downarrow$-closed topologized semilattice, $\Phi : X
\multimap Y$ be a upper semicontinuous $T_1$ multimorphism such
that $\Phi(x) \cap \Phi(y) = \emptyset$ for $x \neq y$. Then the
semilattice $\Phi(X)$ is complete if and only if the semilattice
$\Phi(x)$ is complete for each $x \in X$.
\end{corollary}

\begin{proof}
We show that the inequality $x \leq y$ for $x, y \in X$ implies
$\Phi(x) \cap {\uparrow}\Phi(y) = \emptyset$. Suppose the
opposite. Let $a \in \Phi(x) \cap {\uparrow}\Phi(y)$. This means
that there is $b \in \Phi(y)$ such that $b \leq a$. But then
\linebreak$b = ab \in \Phi(x)\Phi(y) \subset \Phi(xy) = \Phi(x)$
and $b \in \Phi(x) \cap \Phi(y)$, which is a contradiction.
\end{proof}

Theorem \ref{th25} allows to generalize the results on the closure
of semilattices.

\medskip
The {\it Lawson number} $\overline{\Lambda(X)}$ of a Hausdorff
topologized semilattice $X$ is defined as the smallest cardinal
$\kappa$ such that for any distinct points $x,y\in X$ there exists
a family $\mathcal{U}$ of closed neighborhoods of $x$ such that
$|\mathcal{U}|\leq \kappa$ and $\bigcap \mathcal{U}$ is a
subsemilattice of $X$ that does not contain $y$. A topologized
semilattice $X$ is $\omega$-Lawson if and only if it is Hausdorff
and has at most countable Lawson number $\overline{\Lambda(X)}$.

A topological space $X$ is called {\it functionally Hausdorff} if
for any two points $x, y \in X$ there is a continuous real-valued
function $f$ such that $f(x) \neq f(y)$.

A space $X$ is {\it sequential} if for non-closed set $A$ there is
a sequence of elements $A$ converging to some point $x \in
\overline{A} \setminus A$.

In \cite{BaBa1, BaBa2, BaBaGu3} it was proved that a complete subsemilattices of
semitopological functionally Hausdorff (sequential Hausdorff,
$\omega$-Lawson) semilattice are closed. It is also known that a
semitopological semilattice is ${\uparrow}{\downarrow}$-closed
\cite{GiHo}. These results, together with Theorem \ref{th25}, allow us
to formulate:

\begin{corollary}
Let $X$ be a complete topologized semilattice, and let $Y$ be an
$\omega$-Lawson (functionally Hausdorff, sequential)
semitopological semilattice, $\Phi : X \multimap Y$ be an upper
semicontinuous $T_1$ multimorphism such that for any two points
$x, y \in X$ inequality  $x \leq y$ implies $\Phi(x) \cap
{\uparrow}\Phi(y) \subset \Phi(y)$ and the semilattice $\Phi(x)$
is complete for each $x \in X$. Then the set $\Phi(X)$ is closed
in $Y$.
\end{corollary}

We can also reformulate the question of the
closure of semilattices in terms of preimages under certain maps.

\medskip

A point $x$ of a topological space $X$ is called {\it
$\theta$-adherent point} of the set $A \subset X$ if $A \cap
\overline{U}\neq \emptyset$ for any neighborhood $U$ of $x$.

The following concepts was introduced by N.V. Velichko in \cite{Vel}.

$\bullet$ The {$\theta$-closure} of a subset $A$ of a topological
space $X$ is called the set $\thetacl{A} = \lbrace x \in X: x$ is
a $\theta$-adherent point of $A \rbrace$.

$\bullet$ A subset $A$ of a topological space $X$ is called
$\theta$-closed if $\overline{A}^{\theta} = A$.

\begin{theorem}
Let $X$ be a subsemilattice of a topological semilattice  $Y$. If
there exists a closed homomorphism  $h : X \rightarrow E$ from $X$
to a complete topologized semilattice $E$ such that for each $e
\in E$ the set $h^{-1}(e)$ is $\theta$-closed in $Y$, then $X$ is
closed in $Y$.
\end{theorem}

\begin{proof}
Since $h$ is a closed map, the image $h(X)$ is closed in
$E$. Given that completeness is inherited by closed
subsemilattices, we can assume, without loss of generality,
that $h$ is a surjective map. We define a multi-valued map $\Phi :
E \multimap Y$ as follows: $\Phi(e) = h^{-1}(e)$. We show that
$\Phi$ is an upper semicontinuous $T_2$ multimorphism.

Claim that $\Phi^{-1}(F) = h(F)$ for each $F \subset X$.

 Take an element $x \in \Phi^{-1}(F)$. By
definition of the preimage of a set under a multivalued map $\Phi(x) \cap F
= h^{-1}(x) \cap F \neq \emptyset$ and, hence, there is $z \in F$
such that $h(z) = x$ and $x \in h(F)$. Now take $e \in h(F)$ and
find $y \in F$ such that $h(y) = e$. This means that $y \in F \cap
h^{-1}(e)= F \cap \Phi(e)$, so $e \in \Phi^{-1}(F)$; the resulting
inclusions prove the equality $\Phi^{-1}(F) = h(F)$.

Since $h$ is a closed map, $\Phi^{-1}(F) = \Phi^{-1}(F \cap X) =
h(F \cap X)$ is closed in $E$ for every closed set $F \subset Y$,
so $\Phi$ is upper semicontinuous. Since $\Phi(e) = h^{-1}(e)$ is
$\theta$-closed in $Y$, the multi-valued map $\Phi$ has the
property $T_2$.

Now we check that $\Phi$ is a multimorphism. Take $e_1, e_2 \in E$
and $x_1 \in \Phi(e_1) = h^{-1}(e_1), x_2 \in \Phi(e_2) =
h^{-1}(e_2)$. Since $h$ is a homomorphism, $e_1e_2 = h(x_1)h(x_2)
= h(x_1x_2)$ and  $x_1x_2 \in h^{-1}(e_1e_2)$. Then we have that
$\Phi(e_1)\Phi(e_2) = h^{-1}(e_1)h^{-1}(e_2) \subset
h^{-1}(e_1e_2)=\Phi(e_1e_2)$, that is, $\Phi$ is a multimorphism.

To complete the proof, it remains only to note that the closedness of
$\Phi(E) = X$ in $Y$ now follows from Theorem 2.1 in \cite{BB1}.
\end{proof}

\begin{corollary}
Let $X$ be a subsemilattice of a regular semitopological
semilattice $Y$. If there exists a closed homomorphism $h : X
\rightarrow E$ from $X$ into a complete topologizing semilattice
$E$ such that for each $e \in E$ the set $h^{-1}(e)$ is closed in
$Y$, then  $X$ is closed in $Y$.
\end{corollary}

\begin{proof}
It immediately follows from the fact that in regular spaces closure and $\theta$-closure operators coincide.
\end{proof}

\section{The $\delta$-completeness of topologized semilattices}

A point $x$ of topological space $X$ is called {\it
$\delta$-adherent point} of a set $A \subset X$ if $A \cap
Int\overline{U}\neq\emptyset$ for any neighborhood $U$ of $x$.

The \textit{$\delta$-closure} of a subset $A$ of a topological
space $X$ is called the set $\deltacl{A} = \lbrace x \in X: x$ is
a
 $\delta$-adherent point of $A \rbrace$.

A subset $A$ of a topological space $X$ is called $\delta$-closed,
if $\overline{A}^{\delta} = A$.

The concept of $\delta$-closure was introduced by N.V. Velichko in
\cite{Vel}. It is also proved that the intersection and finite union of
$\delta$-closed sets is $\delta$-closed. Obviously, the empty set
and the entire space are $\delta$-closed sets. It follows that for
any topological space $(X, \tau)$ there exists a topology
$\tau_\delta$ such that closed (in $(X, \tau_\delta)$) sets are
exactly $\delta$-closed sets of the space $(X, \tau)$. It is easy to check
that the $\delta$-closure of a set is a $\delta$-closed set.
It follows that $\deltacl{A}$ is the intersection of all
$\delta$-closed sets containing  $A$. Now it is easy to see that
the closure operator in $(X, \tau_\delta)$ is the same as the
$\delta$-closure operator in $(X, \tau)$.

Complements to $\delta$-closed sets are called $\delta$-open sets.

\begin{definition}
A topologized semilattice $X$ is called {\it $\delta$-complete} if
each non-empty chain $C\subset X$ has $\inf{C}$ and $\sup{C}$ that
belong to the $\delta$-closure $\deltacl{C}$ of the chain $C$ in
$X$.
\end{definition}

It follows that a topologized semilattice $(X, \tau)$ is
$\delta$-complete if and only if the semilattice  $(X,
\tau_\delta)$ is complete.

\begin{definition} A topologized semilattice $X$ is called
$\delta$-${\uparrow}{\downarrow}$-\textit{closed} if for each $x
\in X$ sets ${\uparrow}x$ and ${\downarrow}x$ are $\delta$-closed
in $X$.
\end{definition}

\begin{definition}
A {\it weak $\delta$-$chain^\bullet$-topology} on a topologized
semilattice $X$ is called a topology generated by a subbase
consisting of complements to $\delta$-closed chains in $X$.
\end{definition}

\begin{definition}
The topologized semilattice $X$ is called {\it weak
$\delta$-$chain^\bullet$ compact}, if $X$ is compact in its weak
$\delta$-$chain^\bullet$ topology.
\end{definition}

We now formulate an analog of Theorem 1.6 for $\delta$-complete
semilattices.

\begin{theorem}
For a $\delta$-${\uparrow}{\downarrow}$-closed topologized
semilattice $(X, \tau)$ the following conditions are equivalent:

1) $X$ is $\delta$-complete;

2) $X$ is weak $\delta$-$chain^\bullet$ compact;

3) any $\delta$-open cover of $\delta$-closed (in $X$) chain $C$
contains finite subcover.
\end{theorem}

\begin{proof}
Since the closure operator in the semilattice  $(X, \tau_\delta)$
coincides with the $\delta$-closure operator in  $(X, \tau)$, the
semilattice  $(X, \tau_\delta)$ is complete. Recall that closed
sets in $(X, \tau_\delta)$ are exactly  $\delta$-closed sets in
$(X, \tau)$. The statement of this theorem now follows from Theorem 2.3.
\end{proof}

\section{The $\theta$-completeness of topologized semilattices}

\begin{definition}
A topologized semilattice $X$ is called
\textit{$\theta$-complete}, if for each non-empty chain  $C
\subset X$  $\inf{C} \in \thetacl{C}$ and $\sup{C} \in
\thetacl{C}$.
\end{definition}

\begin{definition} A topologized semilattice $X$ is called
\textit{$\theta$-${\uparrow}{\downarrow}$-closed}
(\textit{$\theta$-${\updownarrow}$-closed}), if for any element $x
\in X$ sets ${\uparrow}x$ and ${\downarrow}x$ (set
${\updownarrow}x$) are $\theta$-closed.
\end{definition}

Note that unlike the closure
 and $\delta$-closure operators, the $\theta$-closure operator is not necessarily idempotent, which makes the class of $\theta$-complete semilattices a bit more interesting.

\begin{definition}
Let $X$ be a topologized semilattice and $D \subset X$ is up
(down)-directed. We say that  $D$ {\it up-$\theta$-converges}
(\textit{down-$\theta$-converges}) to the point $x \in X$, if for
any neighborhood $U$ of $x$ there is $d \in D$ such that $D \cap
{\uparrow}d \subset \overline{U}$ ($D \cap {\downarrow}d \subset
\overline{U}$).
\end{definition}

\begin{lemma}\label{up_conv_lemma}
Let $X$ be a $\theta$-complete topologized semilattice. Then any
up-directed set $D \subset X$ up-$\theta$-converges to $\sup{D}$.
\end{lemma}

\begin{proof}
Suppose the opposite. Let $D \subset X$ be a up-directed set does
not up-$\theta$-converge  to $\sup{D}$. Then there exists a
neighborhood $U$ of the point $\sup{D}$ that the set $(D \cap
{\uparrow}d) \setminus \overline{U} \neq \emptyset$ for each $d
\in D$.

We claim that the set $E = D \setminus \overline{U}$ is
up-directed. Indeed, let $e_1, e_2 \in E$; then there is  $d \in
D$ such that $d \geq e_1$ and $d \geq e_2$. Since $(D \cap
{\uparrow}d) \setminus \overline{U} \subset E$, there is  $e' \in
E$ such that $e' \geq d \geq e_1$ and $e' \geq d \geq e_2$, as
required. Note that $\sup{D} = \sup{E}$.

Since $X$ is a $\theta$-complete semilattice, $\sup{E} = \sup{D}
\in \thetacl{E}$. But, on the other hand, $E \cap \overline{U} =
\emptyset$, hence, $\sup{E} \not\in \thetacl{E}$, because
$\sup{E}$ is an inner point of $\overline{U}$.
\end{proof}

The following statement is proved in exactly the same way.

\begin{lemma}\label{down_conv_lemma}
Let $X$ be a $\theta$-complete topologized semilattice. Then any
down-directed set $D \subset X$ down-$\theta$-converges to $\inf
D$.
\end{lemma}

\begin{lemma}\label{chain_lemma}
Let $X$ be a $\theta$-$\updownarrow$-closed semilattice. Then for
any chain $C \subset X$, the set $\thetacl{C}$ is a chain.
\end{lemma}

\begin{proof}
Suppose the opposite, let the set $\thetacl{C}$ contain
incomparable elements $x$ and $y$. Since $x \not\in
{\updownarrow}y$ and ${\updownarrow}y$ is $\theta$-closed, there
is a neighborhood $U$ of the point $x$ such that $\overline{U}
\cap {\updownarrow}y = \emptyset$. Since $x \in \thetacl{C}$,
there is  $z \in \overline{U} \cap C$. Since $z \not\in
{\updownarrow}y$, we have $y \not \in {\updownarrow}z$. Then there
is a neighborhood $V$ of the point $y$ such that $\overline{V}
\cap {\updownarrow}z = \emptyset$, which impossible, since $C
\subset {\updownarrow}z$ and $y \in \thetacl{C}$.
\end{proof}

\begin{theorem}\label{chain_thetacl}
Let $X$ be a $\theta$-complete,
$\theta$-${\uparrow}{\downarrow}$-closed topologized semilattice,
and $C \subset X$ be a chain. Then $\thetacl{C}$ is a
$\theta$-closed set.
\end{theorem}

\begin{proof}
Assume that $\thetacl{C}$ is not $\theta$-closed and there exists
$x \in \overline{C}^{\theta^2} \setminus \thetacl{C}$ (where $\overline{C}^{\theta^2} = \thetacl{\thetacl{C}}$). Note that,
by Lemma \ref{chain_lemma}, the sets $\thetacl{C}$ and
$\overline{C}^{\theta^2}$ are chains. By $\theta$-completeness,
$\inf{C} \in \thetacl{C}$ and $\sup{C} \in \thetacl{C}$. Since $X$
is a $\theta$-${\uparrow}{\downarrow}$-closed semilattice,
$\overline{C}^{\theta^2} \subset {\uparrow}\inf{C} \cap
{\downarrow}\sup{C}$, hence, $\inf{C} \leq x$ and $\sup{C} \geq
x$. By choice $x$, $\inf{C} \neq x$ and $\sup{C} \neq x$, hence,
$\inf{C} < x < \sup{C}$. Note that there are elements $c_1, c_2
\in C$ such that $c_1 < x < c_2$ and, hence, sets ${\downarrow}x
\cap C$ and ${\uparrow}x \cap C$ non empty. Since $X$ is a
$\theta$-complete semilattice, $a:= \sup{({\downarrow}x \cap C)}
\in \thetacl{{\downarrow}x \cap C} \subset \thetacl{C}$ è $b :=
\inf{({\uparrow}x \cap C)} \in \thetacl{{\uparrow}x \cap C}
\subset \thetacl{C}$. By the choice of $x$, the double inequality
 $a < x < b$ is satisfied. Then $x$ does not belong to
 $\theta$-closed set ${\downarrow}a \cup {\uparrow}b$; this means that there is a neighborhood $U$ of the point $x$ such that $\overline{U} \cap
{\downarrow}a \cup {\uparrow}b = \emptyset$. Since $x \in
\overline{C}^{\theta^2}$, there is $y \in \overline{U} \cap
\thetacl{C}$. Since $\thetacl{C}$ is a chain and $y \not\in
{\downarrow}a \cup {\uparrow}b$, we again the double inequality $a
< y < b$. Let us now find a neighborhood $V$ of $y$ such that
$\overline{V} \cap {\downarrow}a \cup {\uparrow}b = \emptyset$ and
element $c \in \overline{V} \cap C$. Since $c \not\in
{\downarrow}a \cup {\uparrow}b$,  $a < c < b$. This inequality
leads us to a contradiction: since $c \in C$, if $c
> a = \sup{({\downarrow}x \cap C)}$ then $c > x$ (otherwise $c \leq a$); similarly, if  $c < b =
\inf{({\uparrow}x \cap C)}$ then $c < x$.
\end{proof}

A subset $M$ of a topological space $X$ is an {\it $H$-set} if
every cover of it by open sets of $X$ has a finite subfamily which
covers $M$ with the closures of its members.

The concept of an $H$-set, which generalizes the concept of a
compact subset of a space, was introduced by N.V. Velichko in
\cite{Vel}.

Recall that a topological space $X$ is called \textit{Urysohn
space}, if for any two distinct points $x, y \in X$ there are
neighborhoods $U_x$, $U_y$ of points $x,y$ such that
$\overline{U_x} \cap \overline{U_y} = \emptyset$.

The following results are well known:

$\bullet$ Every $\theta$-closed subset of an $H$-closed space is
an $H$-set \cite{Vel2}.

$\bullet$ If $X$ is $H$-closed and Urysohn, then $M\subset X$ is
$\theta$-closed if and only if it is an $H$-set \cite{DP}.

$\bullet$ $M$ is an $H$-set of a space $X$ if and only if for
every filter $\mathcal{F}$ on $X$, which meets $M$, $M\cap
ad_{\theta} \mathcal{F}\neq \emptyset$, where $ad_{\theta}
\mathcal{F}=\bigcap \{\overline{F}^{\theta}: F\in \mathcal{F}\}$
\cite{Ha}.

\begin{lemma}
Let $X$ be a $\theta$-complete topologized semilattice. Then every
$\theta$-closed chain $C$ is an $H$-set.
\end{lemma}

\begin{proof}
Let $C$ be a non-empty $\theta$-closed chain in $X$. Consider a
family of open sets $\mathcal{U}$, covering $C$. Let $A$ be a
set of points  $a \in C$ such that the set  ${\downarrow}x \cap C$
can be covered by a finite subfamily $\overline{\mathcal{U}}$.
$A\neq\emptyset$, since it obviously contains the smallest element
$c$ of the chain $C$. By $\theta$-completeness of $X$,  $\sup A
\in \thetacl{A} \subset \thetacl{C} = C$.

We show that $b := \sup A \in A$. Suppose that $b \not\in A$.
Choose $U_b \in \mathcal{U}$ such that $b \in U_b$. By Lemma \ref{up_conv_lemma},
there exists a point $a \in A$ such that $A \cap {\uparrow}a
\subset \overline{U_b}$. By definition of $A$,  $x \in A$ implies
${\downarrow}x \cap C \subset A$. By assumption,  $a < \sup{A} =
b$, so for each $a \leq x \leq b$ there is  $y \in A$ such that $x
\leq y \leq b$, i.e.,  $A \cap {\uparrow}a = (C \cap {\uparrow}a
\cap {\downarrow}b) \setminus \lbrace b \rbrace$. By definition of
$A$, there is a finite subfamily $\mathcal{V} \subset \mathcal{U}$
such that $\bigcup \overline{\mathcal{V}} \supset {\downarrow}a
\cap C$. Let $\mathcal{U'} =  \mathcal{V} \cup \lbrace U_b
\rbrace$. Then $\overline{\mathcal{U'}}$ is the finite cover of
${\downarrow}b \cap C$, and $b \in A$.

Now we claim that $C = {\downarrow}b \cap C$. Suppose that $E := C
\setminus {\downarrow}b \neq \emptyset$. Then $e := \inf{E} \in
\thetacl{E} \subset \thetacl{C} = C$. Note that $a \in
{\downarrow}b \cap C = A$ for every $a \in C$ such that $a < e$.

Choose $U_e \in \mathcal{U}$ such that $e \in U_e$. Obviously that
 $\mathcal{U'} \cup \lbrace \overline{U_e} \rbrace$ is the finite cover of ${\downarrow}e \cap C$, and  $e \in A$. Note that
$b$ is a lower bounded of $C \setminus {\downarrow}b = E$, and $b
\leq e$. Consider two cases:

1) $b < e$. Since $e \in A$, this contradicts the equality  $b =
\sup A$.

2) $b = e$. By Lemma \ref{down_conv_lemma}, there exists $d \in E$ such that $E \cap
{\downarrow}d \subset \overline{U_b}$. Note that $d
> e = b$ ($e \not\in E$). Then ${\downarrow}d \cap C \subset
\bigcup \mathcal{\overline{U'}}$ and $d \in A$, which contradicts
the equality $b = \sup A$.
\end{proof}

Combining this Lemma with Theorem \ref{chain_thetacl}, we obtain

\begin{corollary}\label{H_set_cor}
Let $X$ be a $\theta$-complete
$\theta$-${\uparrow}{\downarrow}$-closed topologized semilattice.
Then for each chain $C \subset X$ the chain $\thetacl{C}$ is an
$H$-set.
\end{corollary}

\begin{lemma}\label{H_set_lemma2}
Let $X$ be a $\theta$-${\uparrow}{\downarrow}$-closed semilattice
in which $\thetacl{C}$ is an $H$-set for any chain $C \subset X$.
Then $X$ is a $\theta$-complete semilattice.
\end{lemma}

\begin{proof}
Let $C \subset X$ be a chain. By Lemma \ref{chain_lemma}, $\thetacl{C}$ is also
chain. Consider a centered family $\mathcal{F}_< = \lbrace
\thetacl{C} \cap {\downarrow}x : x \in C \rbrace$ subsets
$\thetacl{C}$. Since the chain $\thetacl{C}$ is an $H$-set and
$\thetacl{\thetacl{C} \cap {\downarrow}x} \subset
\thetacl{{\downarrow}x} =  {\downarrow}x$, we have
$\bigcap \mathcal{F}_< = \bigcap\limits_{x \in C} \thetacl{C} \cap
{\downarrow}x \supset \bigcap\limits_{x \in C}
\thetacl{\thetacl{C} \cap {\downarrow}x} \cap \thetacl{C} \neq
\emptyset$ by the criterion of the H-set mentioned above and proved in \cite{Ha}. Obviously, $\bigcap \mathcal{F}_<$ contains the only
element $c$ that is the smallest in $\thetacl{C}$. Clearly, $c
\leq x$ for each $x \in C$; suppose that there is  $a < c$, and
$a=inf C$. Then  $c \not\in {\uparrow}a \supset \thetacl{C}$,
which contradicts the choice  $c$. Thus, $c = \inf{C} \in
\thetacl{C}$; similarly, it can be shown that the intersection of
a centered family $\mathcal{F}_> = \lbrace \thetacl{C} \cap
{\uparrow}x : x \in C \rbrace$ contains $\sup C$.
\end{proof}

By Corollary \ref{H_set_cor} and Lemma \ref{H_set_lemma2} we have the following
theorem.

\begin{theorem}
A $\theta$-${\uparrow}{\downarrow}$-closed topologized semilattice
$X$ is $\theta$-complete if and only if for any chain $C \subset
X$ the chain $\thetacl{C}$ is an $H$-set.
\end{theorem}

We show that a slightly more general result can be obtained from
other arguments. We give the necessary definitions.

\begin{definition}
Let $X$ be a topologized semilattice. The {\it weak
$\theta$-$chain^\bullet$-topology} on $X$ is called a topology
generated by a subbase consisting of complements to
$\theta$-closures of chains in $X$.
\end{definition}

\begin{definition}
A topologized semilattice $X$ is called {\it weak
$\theta$-chain$^\bullet$ compact} if $X$ is compact in its weak
$\theta$-$chain^\bullet$ topology.
\end{definition}

\begin{lemma}
Let $(X, \tau)$ be a $\theta$-complete
$\theta$-${\uparrow}{\downarrow}$-closed topologized semilattice.
Then $X$ is weak $\theta$-$chain^\bullet$ compact.
\end{lemma}

\begin{proof}
In \cite{Vel}, it was proved that the intersection and finite
union of $\theta$-closed sets is $\theta$-closed. Obviously, the
empty set and the entire space are $\theta$-closed sets. It
follows that for any topological space $(X, \tau)$ there exists a
topology $\tau_\theta$ such that closed (in $(X, \tau_\theta)$)
sets are exactly $\theta$-closed sets of the space $(X, \tau)$.

 Clearly, the $\theta$-closure of a set $A \subset
X$ in $(X, \tau)$ is contained in the closure of $A$ in $(X,
\tau_\theta)$. It follows that a semilattice $(X, \tau_\theta)$ is
complete and ${\uparrow}{\downarrow}$-closed. By Lemmas \ref{chain_lemma} and \ref{chain_thetacl},
for each chain $C$  the chain $\thetacl{C}$ is $\theta$-closed. It
follows that the weak $\theta$-chain$^\bullet$-topology on $X$ is
contained in the weak$^\bullet$ $\mathcal{W}^\bullet$-topology of
the space $(X, \tau)$. By Lemma 5.4 of \cite{BaBa4}, the semilattice $(X,
\tau_\theta)$ is compact in its weak$^\bullet$
$\mathcal{W}^\bullet$-topology. Moreover, $(X, \tau)$ is weak
$\theta$-chain$^\bullet$ compact.
\end{proof}

\begin{lemma}
Let $X$ be a weak $\theta$-chain$^\bullet$ compact,
$\theta$-${\updownarrow}$-closed topologized semilattice. Then
$\thetacl{C}$ is an $H$-set for any chain $C \subset X$.
\end{lemma}

\begin{proof}
Let $C \subset X$ be a chain. We show that for any centered family
$\lbrace F_\alpha \rbrace_{\alpha \in \Lambda}$ of subsets
$\thetacl{C}$, $\bigcap\limits_{\alpha \in \Lambda}
\thetacl{F_\alpha} \cap \thetacl{C} \neq \emptyset$. By Lemma \ref{chain_lemma},
the set $\thetacl{C}$ is a chain; then so is
$F_\alpha$. This means that the sets $\thetacl{F_\alpha}$ closed
in weak $\theta$-$chain^\bullet$ topology on $X$.
\end{proof}

\begin{theorem}
For a $\theta$-${\uparrow}{\downarrow}$-closed semilattice $X$ the
following statements are equivalent:

1) $X$ is  $\theta$-complete;

2) $X$ is $\theta$-$chain^\bullet$ compact;

3) $\thetacl{C}$ is an $H$-set for any chain $C \subset X$.
\end{theorem}

\begin{lemma}
Let $X$ be a Urysohn space, $r : X \rightarrow X$ be a retraction.
Then the set $r(X)$ is $\theta$-closed in $X$.
\end{lemma}

\begin{proof}
Suppose that $\thetacl{r(X)} \setminus r(X) \neq \emptyset$ and
let $x \in \thetacl{r(X)} \setminus r(X)$. Note that $r(x) \neq
x$. Find neighborhoods  $U_x, U_{r(x)}$ of $x$ and $r(x)$,
respectively, such that $\overline{U_x} \cap \overline{U_{r(x)}} =
\emptyset$. By the continuity of $r$, there is a neighborhood $V_x
\subset U_x$ of $x$ such that $r(V_x) \subset U_{r(x)}$. Clearly,
then $\overline{V_x} \cap \overline{U_{r(x)}} = \emptyset$. Since
$x \in \thetacl{r(X)}$, $\overline{V_x} \cap r(X) \neq \emptyset$.
Let $z \in \overline{V_x} \cap r(X)$. Since $r(z) = z$ and $r$ is
continuous,  $z \in r(\overline{V_x}) \subset
\overline{r(V_x)}^{r(X)} \subset \overline{U_{r(x)}}^{r(X)}
\subset \overline{U_{r(x)}}$, that is $z \in \overline{V_x} \cap
\overline{U_r(x)}$, contradiction.
\end{proof}

\begin{proposition}
An Urysohn semitopological semilattice $X$ is
$\theta$-${\uparrow}{\downarrow}$-closed.
\end{proposition}

\begin{proof}
Consider an element $x \in X$ and the mapping $s_x : X \rightarrow
X, s_x : y \mapsto xy$. Since $X$ is Hausdorff, $\lbrace x
\rbrace$ is $\theta$-closed. Is easy to check, that the preimage of a
$\theta$-closed set under continuous mapping is $\theta$-closed.
It is easy to see that ${\uparrow}x = s_x^{-1}(x)$ and $s_x$ is a
retraction $X$ on ${\downarrow}x$; $\theta$-closedness
${\uparrow}x$ and ${\downarrow}x$ follows from the continuity
of $s_x$.
\end{proof}

\section{The $\Theta$-completeness of topologized semilattices}

Unlike the $\delta$-closure, the operation of
taking the $\theta$-closure of a set is not necessarily idempotent, that is, the $\theta$-closure
of set may not be $\theta$-closed. This fact motivates the
following definition.

\begin{definition}
A $\Theta$-closure of a subset $A$ of a topological space $X$ is
called the set $\sthetacl{A}$, equal to the intersection of all
$\theta$-closed subsets of $X$, containing $A$.
\end{definition}

In \cite{Vel}, it was proved that the intersection and finite union of
$\theta$-closed sets is $\theta$-closed. Obviously, the empty set
and the entire space are $\theta$-closed sets. It follows that for
any topological space $(X, \tau)$ there exists a topology
$\tau_\theta$, such that the closed sets in space $(X,
\tau_\theta)$ exactly the $\theta$-closed sets in  $(X, \tau)$.
Obviously, the closure operator in $(X, \tau_\theta)$ is the same
as the $\Theta$-closure operator in $(X, \tau)$.

Complement to a $\theta$-closed set is called a $\theta$-open set.

\begin{definition}
A topologized semilattice $X$ is called \textit{$\Theta$-complete}
if for every non-empty chain $C \subset X$ $\inf{C} \in
\sthetacl{C}$ and $\sup{C} \in \sthetacl{C}$.
\end{definition}

\begin{definition}
Let $X$ be a topologized semilattice. A \textit{weak
$\Theta$-$chain^\bullet$ topology} on $X$ is called a topology
generated by subbase consisting of the complements of
$\theta$-closed chains in $X$.
\end{definition}

\begin{definition}
The topologized semilattice $X$ is called {\it weak
$\Theta$-$chain^\bullet$ compact} if $X$ is compact in its weak
$\Theta$-$chain^\bullet$ topology.
\end{definition}

\begin{theorem}
For a $\theta$-${\uparrow}{\downarrow}$-closed topologized
semilattice $X$ the following conditions are equivalent:

 1) $X$ is $\Theta$-complete;

2) $X$ is  weak $\Theta$-$chain^\bullet$ compact;

3) any $\theta$-open cover $\lbrace U_\alpha \rbrace_{\alpha \in
\Lambda}$ of a $\theta$-closed (in $X$) chain $C$ contains a
finite subcover.
\end{theorem}

\section{The relationship of different types of completeness}

In this chapter, we will look at the relationships of the classes
of topologized semilattices that we introduced earlier.

In \cite{Vel}, it was proved that for an arbitrary topological space $X$
and its subset $A$ the following inclusions hold: $\overline{A}
\subset \deltacl{A} \subset \thetacl{A}$. From the definitions, it
is clear that the same is always true for  $\thetacl{A} \subset
\sthetacl{A}$. In addition, the same paper shows that the
operators $\delta$-closures and $\theta$-closures coincide in the
class of regular spaces.

\begin{proposition}
Let $X$ be a topologized semilattice. Then

1) if $X$ is complete, then it is $\delta$-complete;

2) if $X$ is $\delta$-complete, then it is $\theta$-complete;

3) if $X$ is $\theta$-complete, then it is $\Theta$-complete.
\end{proposition}

\begin{proposition}
For a regular topologized semilattice $X$ the following conditions
are equivalent:

1) $X$ is complete;

2) $X$ is $\delta$-complete;

3) $X$ is $\theta$-complete;

4) $X$ is $\Theta$-complete.
\end{proposition}

\begin{proposition}
Let $X$ be a topologized semilattice.

1). If $X$ is weak chain$^\bullet$ compact, it is also weak
$\delta$-chain$^\bullet$ compact.

2). If $X$ is a weak $\delta$-chain$^\bullet$ compact or weak
$\theta$-chain$^\bullet$-compact, it is weak
$\Theta$-chain$^\bullet$-compact. If, in addition, $X$ is
$\theta$-${\updownarrow}$-closed, then weak
$\delta$-chain$^\bullet$-compactness implies weak
$\theta$-chain$^\bullet$-compactness.
\end{proposition}

\begin{proof}
The first part of the statement follows from the fact that
$\theta$-closed chain is $\delta$-closed, and, hence, it is
closed; thus weak $\Theta$-chain$^\bullet$-topology is contained
in weak $\delta$-chain$^\bullet$-topology and weak
$\theta$-chain$^\bullet$-topology. Note that weak
$\delta$-chain$^\bullet$-topology  weaker than weak
chain$^\bullet$-topology. It remains to note that compactness in
some topology implies compactness in any weaker topology.

The second part of the statement follows from the fact that in a
$\theta$-${\updownarrow}$-closed semilattice the set $\thetacl{C}$
is a $\delta$-closed chain.
\end{proof}

Finally, we show how the various types of compactness of chains
are related.

\begin{proposition}
Let $(X, \tau)$ be a topologized semilattice. Then

1) if every closed in $(X,\tau)$ chain is compact (i.e. $X$ is
chain-compact) then every $\delta$-closed chain is compact in $(X,
\tau_\delta)$;

2) if every $\delta$-closed chain is compact in $(X,\tau_\delta)$
then every $\theta$-closed chain is compact in $(X,\tau_\theta)$;

3) if $\thetacl{C}$ is an $H$-set for any chain $C \subset X$ then
every $\theta$-closed chain is compact in $(X, \tau_\theta)$.

Moreover, if $X$ is a $\theta$-${\updownarrow}$-closed
subsemilattice then

4) if every $\delta$-closed chain is compact in $(X,\tau_\delta)$
then $\thetacl{C}$ is an $H$-set for any chain $C \subset X$.
\end{proposition}

\begin{proof}
Points 1) and 2) are obvious. We prove point 3).

Let $C \subset X$ be a $\theta$-closed (in $X$) chain and $\lbrace
U_\alpha \rbrace_{\alpha \in \Lambda}$  be a family of
$\theta$-open (in $X$) sets covering $C$. By definition of
$\theta$-open set, for each point $x \in C \cap U_\alpha$ there
exists a neighborhood $U_x$ such that $\overline{U_x} \cap (X
\setminus U_\alpha) = \emptyset$ that is  $\overline{U_x} \subset
U_\alpha$. Since $\thetacl{C} = C$, the set $C$ is a $H$-set.
Since $\lbrace U_x \rbrace_{x \in C}$ is an open cover of $C$,
there is a finite set $x_1, ..., x_n$ such that $C \subset
\bigcup\limits_{i=1}^{n} \overline{U_{x_i}}$. For every $x_i$
there is $\alpha_i \in \Lambda$ such that $\overline{U}_{x_i}
\subset U_{\alpha_i}$. It follows that $\lbrace U_{\alpha_i}
\rbrace$ is a finite cover of $C$.

Now we prove 4). Consider a chain $C \subset X$. Since $X$ is
$\theta$-${\updownarrow}$-closed, the set $\thetacl{C}$ is a
$\delta$-closed chain. Let $\lbrace U_\alpha \rbrace_{\alpha \in
\Lambda}$ is a cover of $C$. It is known [1], that the sets
$Int{\overline{U}}$ are $\delta$-open. Then there is a finite set
$\alpha_1, ..., \alpha_n$ such that $C \subset
\bigcup\limits_{i=1}^{n} Int{\overline{U_{\alpha_i}}}$. Then $C
\subset \bigcup\limits_{i=1}^{n} \overline{U_{\alpha_i}}$.
\end{proof}

\bigskip

\begin{center}

\ingrw{130}{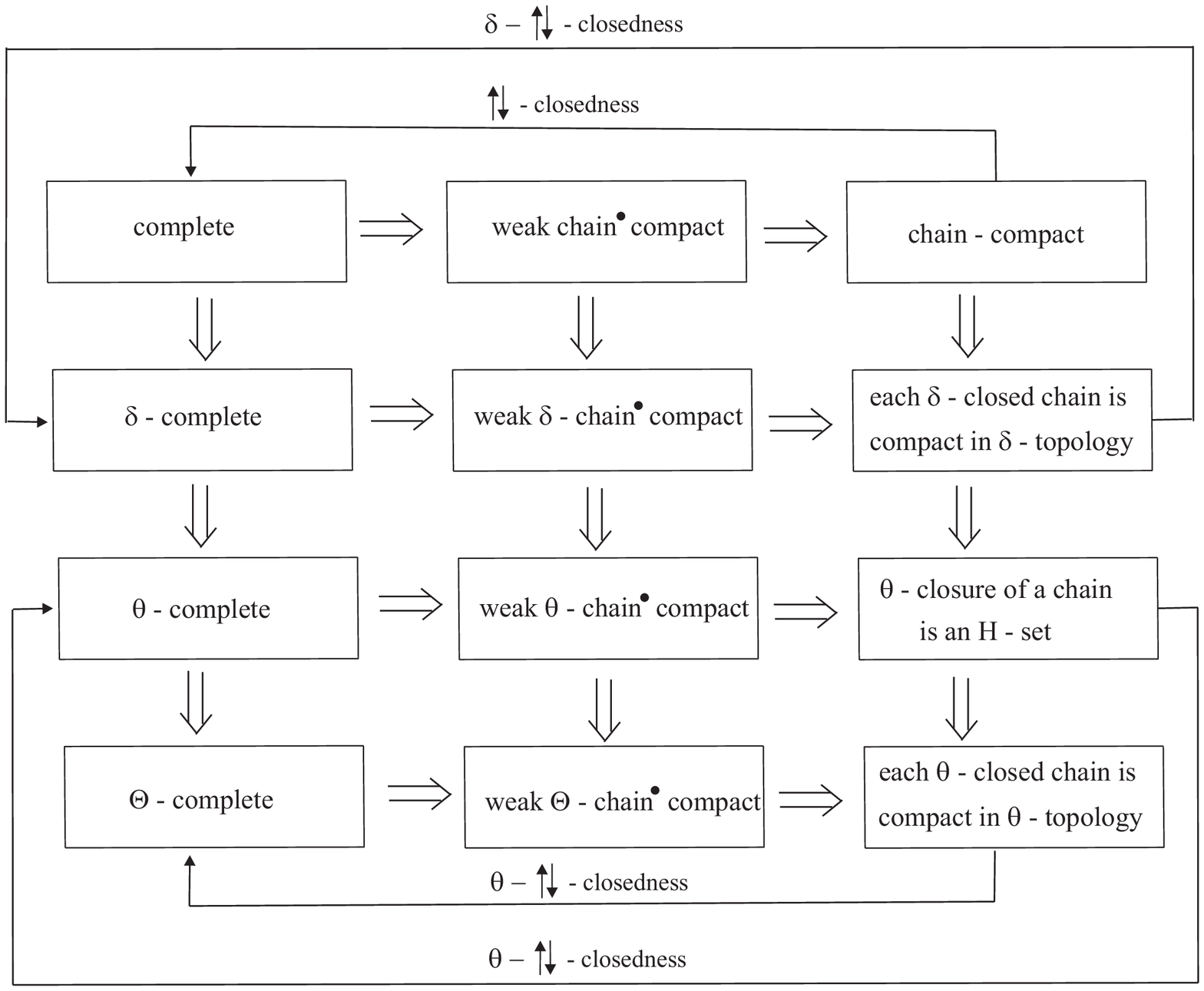}

\medskip

Diagram 1. Implications between various types of completeness of
semitopological semilattices

\end{center}

\bigskip

\section{Examples}

In conclusion, we show examples that separate the classes we have
introduced.

\begin{example}
There exists a $\delta$-complete topologized semilattice such that
it is not complete.

Consider the segment $I=[0,1]$, ordered by the natural order. We
define on $I$ the topology $\tau$ in terms of fundamental systems
of neighborhoods $\mathcal{B}(x)$:

$$  \mathcal{B}(x)= \left\{
\begin{array}{lcr}
   (x - \varepsilon, x + \varepsilon) \cap I : \varepsilon > 0 , & \text{$x \neq 0$}.\\
   ([0, \varepsilon) \cap I) \setminus \{ \frac{1}{n} : n\in \mathbb{N}\} : \varepsilon > 0 , & \text{$x = 0$}.\\
  \end{array}
\right.
$$

\medskip

Note that $Int(\overline{(x - \varepsilon, x + \varepsilon) \cap
I}) = (x - \varepsilon, x + \varepsilon) \cap I$ and
$Int(\overline{([0, \varepsilon) \cap I) \setminus \{\frac{1}{n}:
n\in \mathbb{N} \}}) = [0,\varepsilon) \cap I$, i.e.
$\delta$-topology on $(I, \tau)$ coincide with the Euclidean
topology on $I$, hence, $(I, \tau)$ is $\delta$-complete.

Since $0 \not\in \overline{\{\frac{1}{n} : n\in \mathbb{N}\}}$
then $(I,\tau)$ is not complete.
\end{example}

\begin{example} There exists a $\theta$-complete topologized semilattice such that it is not
$\delta$-complete.

Let $X=[0,1]\times[0, 1]$. Denote by $B(x, \varepsilon)$ --- an
open $\varepsilon$-ball of $x$ in $\mathbb{R}^2$, $B[x,
\varepsilon]$ -- a closed $\varepsilon$-ball of $x$. We define the
topology on $X$ in terms of fundamental systems of neighborhoods
$\mathcal{B}(x)$:

$$ \mathcal{B}(x)=\left\{
\begin{array}{lcr}
    \lbrace B(x, \varepsilon) \cap X : \varepsilon > 0 \rbrace, \text{$x = (a, b), b \neq 0$}.\\
    \lbrace (B(x, \varepsilon)) \cap X) \setminus \lbrace 0 \rbrace \times (a-\varepsilon, a+\varepsilon) \cup \lbrace x \rbrace : \varepsilon > 0 , \text{$x = (a, 0)$}.
 \end{array}
\right.
$$

Define the operation on $X$:

$$ xy=\left\{
\begin{array}{lcr}
      x, \text{$x = y \vee x = (a_1, 0), y = (a_2, 0), a_1 < a_2$}.\\
    (0, 0), \text{otherwise}.
  \end{array}
\right.
$$
Note that $X$ is a semilattice in which the only maximum infinite
chain is the set $C = [0, 1] \times \lbrace 0 \rbrace$, order
isomorphic to segment $[0, 1]$. Consider a chain $A \subset C$. By
order isomorphism, there are $\sup{A} \in C$ and $\inf{A} \in C$.
Clearly, the closure of the neighborhood $U = (B((0, a),
\varepsilon)) \cap X) \setminus \lbrace 0 \rbrace \times
(a-\varepsilon, a+\varepsilon) \cup \lbrace x \rbrace$ of the
point $(0, a)$ is the set $B[a, \varepsilon] \cap X$, i.e.,
$\overline{U} \cap C = \lbrace 0 \rbrace \times [a - \varepsilon,
a + \varepsilon]$. By definitions $\sup$ and $\inf$, we have that
$\sup{A} \in \thetacl{A}$ and $\inf{A} \in \thetacl{A}$. But, $X$
is not $\delta$-complete semilattice: for the chain $A =
(\frac{1}{3}, \frac{2}{3}) \times \lbrace 0 \rbrace$ we have
$\sup{A} = \frac{1}{3} \times \lbrace 0 \rbrace, \inf{A} =
\frac{2}{3} \times \lbrace 0 \rbrace$, but $U=Int\overline{U}$ for
a base neighborhood $U$ of the point $(\frac{1}{3}, 0)$ and $U
\cap A = \emptyset$.
\end{example}

\begin{example}
There exists a $\Theta$-complete topologized semilattice such that
it is not $\theta$-complete.

Let $A_0 = \lbrace 0 \rbrace \times \omega_1, A_1 = \lbrace 1
\rbrace \times \omega_1, A_2 = \lbrace 2 \rbrace \times (\omega_1
+ 1)$. For an ordinal number $\alpha$ define the sets
$\lim^2(\alpha) := \lbrace \beta \in \lim(\alpha) : \beta =
\sup(\lim(\beta)) \rbrace$ and $\lim^1(\alpha) := \lim(\alpha)
\setminus \lim^2(\alpha)$. Clearly, the sets $\lim(\omega_1),
\lim^1(\omega_1)$ and $\lim^2(\omega_1)$ are unbounded in
$\omega_1$, and because they has the same order type; fix
isomorphisms  $f_1 : \lim(\omega_1) \rightarrow \lim^1(\omega_1)$
and $f_2 : \lim^2(\omega_1) \rightarrow \lim^2(\omega_1)$ of
well-ordered sets.

Let $A = A_0 \sqcup A_1 \sqcup A_2$.  Define on $A$ the following
equivalence relation: $x \sim y$ if and only if $x = y$ or $x =
(0, \alpha), y = (1, f_1(\alpha))$ where $\alpha \in
\lim(\omega_1)$ or $x = (1, \alpha), y = (2, f_2(\alpha))$ where
$\alpha \in \lim^1(\omega_1)$.

\bigskip

\begin{center}

\ingrw{100}{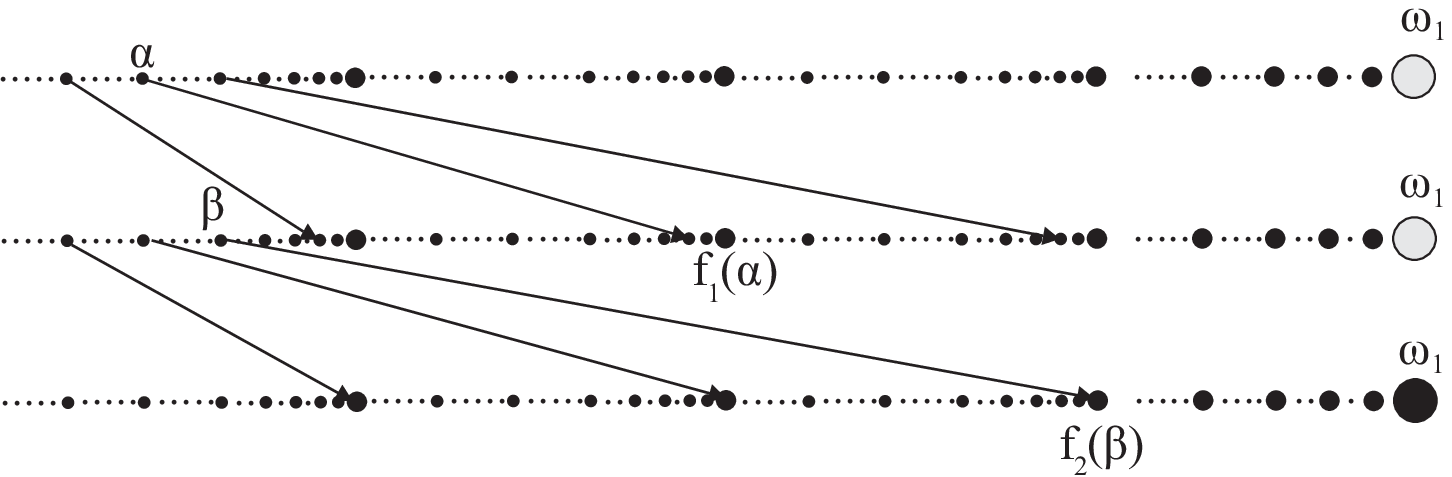}

Diagram 2. Example 7.3
\end{center}

Define on the $Y = A/{\sim}$ the following operation:

$$ \lbrack x \rbrack \cdot \lbrack y \rbrack =\left\{
\begin{array}{lcr}
      \lbrack x \rbrack, \, \, \, \lbrack x \rbrack=\lbrack y \rbrack.\\
    \lbrack(0,\min\{\alpha,\beta\}) \rbrack,  \, \, \, (0,\alpha)\in \lbrack x \rbrack, (0,\beta) \in \lbrack y \rbrack.\\
    \lbrack x \rbrack, \, \, \, \lbrack y \rbrack=\lbrack (2,\omega_1) \rbrack.\\
    \lbrack y \rbrack, \, \, \, \lbrack x \rbrack=\lbrack (2,\omega_1) \rbrack.\\
   \lbrack (0,0) \rbrack, \, \, \, \text{otherwise}.\\
\end{array}
\right.
$$

Note that  $(Y, \cdot)$ is semilattice in which only maximum
infinite chain is the set $Z = \lbrace [(0, \alpha)] : \alpha \in
\omega_1 \rbrace \cup \lbrace [(2, \omega_1)] \rbrace$.

The set  $\lbrace [(i, \beta)] : \beta > \gamma$ and $\beta \leq
\alpha \rbrace$ and $\lbrace [(i, \beta) : \beta > \gamma$ and
$\beta \leq \alpha$ and $\beta \in I(\omega_1)]$, where
$i=\overline{0, 2}$ denote by $(\gamma, \alpha]^i$ and $(\gamma,
\alpha]_I^i$, respectively.  We define the topology on $Y$ in
terms of fundamental systems of neighborhoods $\mathcal{B}(x)$:

$$  \mathcal{B}(x)=\left\{
\begin{array}{lcr}
    \{\{x\}\}, \text{$x = [(i, \alpha)], i=\overline{0, 2}, \alpha \in I(\omega_1)$}.\\
    \{ (\beta, \alpha]^0 \cup (\gamma, f_1(\alpha)]_I^1 : \beta < \alpha, \gamma < f_1(\alpha) \}, \text{$x = [(0, \alpha)], \alpha \in \lim(\omega_1)$}.\\
    \{ (\beta, \alpha]^1 \cup (\gamma, f_2(\alpha)]_I^2 : \beta < \alpha,  \gamma < f_2(\alpha) \}, \text{$x = [(1, \alpha)], \alpha \in \lim^2(\omega_1)$}.\\
    \{ (\beta, \alpha]_I^2 : \beta < \omega_1 \}, \text{$x = [(2, \alpha)], \alpha = \omega_1 \vee \alpha \in
    \lim^1(\omega_1)$}.\\
 \end{array}
\right.
$$

We show that the semilattice $Y$ is $\Theta$-complete. Let $C
\subset Y$ be a chain. If $C$ is finite, then it contains infimum
and supremum. Otherwise $C \subset Z$. Since the chain $Z$
well-ordered by the natural order, $C$ contains infimum. If the
set $E := \lbrace \alpha \in \omega_1 : [(0, \alpha)] \in C
\rbrace$ is bounded in $\omega_1$, then $\sup{C} = [(0, \sup{E})]
\in \overline{C} \subset \sthetacl{C}$. If $E$ is unbounded then
$\sup{C} = [(2, \omega_1)]$. Note that $B = \lbrace [(0, \alpha)]
: \alpha \in \omega_1 \rbrace$ (in subspace topology) is
homeomorphic to $\omega_1$. It follows that the set
$\overline{C}^B \cap \lbrace [(0, \alpha)] : \alpha \in
\lim(\omega_1) \rbrace$ is of power $\omega_1$ (since it is the
intersection of closed unbounded sets). Then $\overline{C} \cap
\lbrace [(1, \alpha)] : \alpha \in \lim^1(\omega_1)\rbrace$ is
also of power $\omega_1$ (since $\overline{C}^B \subset
\overline{C}$, $[(0, \alpha)] = [1, f_1(\alpha)]$ for limit
ordinals $\alpha < \omega_1$ and $f_1$ is bijection). Then
the set $\thetacl{\overline{C}} \cap \lbrace [(1, \alpha)] : \alpha \in
\lim^2(\omega_1)\rbrace \subset \lb \subset \overline{C}^{\theta^2} \cap \lbrace [(1, \alpha)] : \alpha \in
\lim^2(\omega_1)\rbrace$ is of power $\omega_1$, since the closure
of a standard neighborhood of the point $[(1, \alpha)]$, where
$\alpha \in \lim^2(\omega_1)$, contains an infinite number of
points $[(1, \beta)]$, where $\beta \in \lim^1(\omega_1)$. Note
that $\overline{(\beta, \omega_1]_I^2} \supset (\beta, \omega_1]^2$, so it is easy to see that $[(2, \omega_1)] \in \overline{C}^{\theta^2} \subset
\sthetacl{C}$.

Note that if $C = \lbrace [(0 ,\alpha] : \alpha \in I(\omega_1)
\rbrace$, then $[(2, \omega_1)] \in \overline{C}^{\theta^2}$, but
$[(2, \omega_1)] \not\in \thetacl{C}$, i.e.,  $Y$ is not
$\theta$-complete semilattice.
\end{example}

\bibliographystyle{model1a-num-names}
\bibliography{<your-bib-database>}

\begin{thebibliography}{10}

\bibitem{BB1}
Banakh T., Bardyla S., \textit{Characterizing chain-finite and
chain-compact topological semilattices}, Semigroup Forum (2018).

\bibitem{BaBa7}
Banakh T., Bardyla S., \textit{Completeness and absolute
$H$-closedness of topological semilattices}, Topology Appl. {\bf
260} (2019), 189--202.

\bibitem{BaBa1}
Banakh T., Bardyla S., \textit{On images of complete topologized
subsemilattices in sequential semitopological semilattices},
Semigroup Forum. {\bf 100}:3, (2020), 662--670.

\bibitem{BaBa2}
Banakh T., Bardyla S., Ravsky A. \textit{The closedness of
complete subsemilattices in functionally Hausdorff semitopological
semilattices}, Topology and its Applications, {\bf 267} (2019)
106874.

\bibitem{BaBaGu3}
Banakh T., Bardyla S., Gutik O. \textit{The Lawson number of a
semitopological semilattice}, Semigroup Forum, (2021).
https://doi.org/10.1007/s00233-021- 10184-z

\bibitem{BaBa4}
Banakh T., Bardyla S. \textit{The interplay between weak
topologies on topological semilattices}, Topology and its
Applications, {\bf 259}, (2019), 134--154.

\bibitem{BaBa5}
Banakh T., Bardyla S. \textit{Complete topologized posets and
semilattices}, preprint. URL: https://arxiv.org/abs/1806.02869


\bibitem{BaBa6}
Banakh T., Bardyla S. \textit{Characterizing chain-compact and
chain-finite topological semilattices}, Semigroup Forum, {\bf 98},
(2019), 234--250.

\bibitem{DP}
Dickman R.F., Porter J.R. \textit{$\theta$-closed subsets of
Hausdorff spaces}, Pac. J. of Math. 59 (1975), 407--415.




\bibitem{Ha}
Hamlett T., \textit{$H$-closed spaces and the associated
$\theta$-convergence spaces}, Math. Chronicle 8 (1979), 83--88.


\bibitem{Eng}
Engelking R., \textit{General Topology, Revised and completed
edition}. Heldermann Verlag Berlin, 1989.


\bibitem{GiHo}
Gierz G., Hofmann K. H., Keimel K., Lawson J. D., Mislove M. W.,
Scott D. S., \textit{Continuous Lattices and Domains}, Cambridge:
Cambridge Univ. Press, 2003.

\bibitem{Vel}
Velichko N.V. \textit{$H$-closed topological spaces}, Mat. Sb.
(N.S.) {\bf 70}(112):1 (1966), 98--112 (in Russian).


\bibitem{Vel2}
Velichko N.V. \textit{$H$-closed topological spaces}, Am. Math.
Soc. Transl. 78(2) (1968), 103--118.

\bibitem{stepp}
Stepp J. W. \textit{Algebraic maximal semilattices}, Pacific J.
Math. {\bf 58}:1 (1975), 243--248.

\bibitem{gutik1}
O. V. Gutik, M. Radjagopalan, and K. Sundaresan, Compact semilattices with open principal filters, J. Austral. Math. Soc. 72, №3 (2002), 349-362 (doi: 10.1017/S1446788700036776, MR 1902205 (2003f:22005), Zbl 1024.22001).

\bibitem{gutik2}
I. Chuchman, and O. Gutik, On H-closed topological semigroups and semilattices, Algebra Discr. Math. (2007), no. 1, 13-23 (MR2367511 (2008j:22003), Zbl 1164.06332).

\bibitem{gutik3}
O. Gutik, and D. Repovš, On linearly ordered H-closed topological semilattices, Semigroup Forum 77 (2008), no. 3, 474-481 (doi: 10.1007/s00233-008-9102-4, MR2457332 (2009i:54063), Zbl 1166.22003, arXiv:0804.1438).

\end{thebibliography}

\end{document}